\documentclass[11pt]{amsart}
\usepackage[latin1]{inputenc}
\usepackage{amsmath}
\usepackage{amsfonts}
\usepackage{amssymb}
\usepackage{graphicx}
\usepackage{fourier}
\usepackage{hyperref}
\usepackage{enumerate}
\usepackage{esint}
\usepackage{bm}
\usepackage{ulem}
\usepackage{xcolor} 
\usepackage{verbatim}        
\usepackage{bbm}

\usepackage{mathrsfs}

\newtheorem{theorem}{Theorem}[section]
\newtheorem{lemma}[theorem]{Lemma}

\theoremstyle{definition}
\newtheorem{definition}[theorem]{Definition}

\theoremstyle{remark}
\newtheorem{remark}[theorem]{Remark}

\newcommand{\inner}[1]{\left\langle#1\right\rangle}
\newcommand{\norm}[1]{\left\lVert#1\right\rVert}
\newcommand{\abs}[1]{\left\lvert#1\right\rvert}
\newcommand{\pa}[1]{\left( #1 \right)}

\newcommand{\br}[1]{\left\lbrace #1\right\rbrace}
\newcommand{\R}{\mathbb{R}}

\newcommand{\N}{\mathbb{N}}

\newcommand{\PP}{\mathbf{P}}

\newcommand{\W}{\mathbf{W}}
\numberwithin{equation}{section}

\def\1{\mathbbm{1}}

\definecolor{bostonuniversityred}{rgb}{0.8, 0.0, 0.0}
 
\definecolor{byzantium}{rgb}{0.44, 0.16, 0.39}

\newcommand{\D}{{\bf D}}

\newcommand{\II}{{\bf I}}
\newcommand{\C}{{\bf C}}

\newcommand{\wt}[1]{\widetilde{#1}}
\DeclareMathOperator{\dive}{{\mathrm{div}}}
\DeclareMathOperator{\diag}{{\mathrm{diag}}}
\newcommand{\conA}{conditions \eqref{item:cond_semipositive}-\eqref{item:cond_no_invariant_subspace_to_kernel}}

\newcommand{\h}{L^2\left(\mathbb{R}^d,f_\infty^{-1}\right)}
\DeclareMathOperator{\linspan}{{\mathrm{span}}}
\DeclareMathOperator{\re}{{\mathrm{Re}}}

\title{Sharp Decay of the Fisher Information for Degenerate Fokker--Planck Equations}
\author{Anton Arnold$^1$, Amit Einav$^2$ \& Tobias W\"ohrer$^3$} \date{\today}
\address{$^1$Vienna University of Technology, Institute of Analysis and Scientific Computing, Wiedner Hauptstr. 8-10, A-1040 Wien, Austria}
\email{anton.arnold@tuwien.ac.at}

\address{$^2$Durham University, School of Mathematical Sciences, Upper Mountjoy Campus, Stockton Road, DH1 3LE, Durham, United Kingdom}
\email{amit.einav@durham.ac.uk}
\address{$^3$Technical University of Munich, Faculty of Mathematics, Boltzmannstra\ss e 3, 85748 Garching bei M\"unchen}
\email{tobias.woehrer@tum.de}
\thanks{The first author was supported by the FWF-funded SFB \#F65. 
	The third author was supported by the FWF under grant no. J 4681-N}

\begin{document}
\date{}
\maketitle

\begin{abstract}
The goal of this work is to find the sharp rate of convergence to equilibrium under the quadratic Fisher information functional for solutions to Fokker-Planck equations governed by a constant drift term and a constant, yet possibly degenerate, diffusion matrix. A key ingredient in our investigation is a recent work of Arnold, Signorello, and Schmeiser, \cite{ASS} where the $L^2$-propagator norm of such Fokker-Planck equations was shown to be identical to the propagator norm of a finite dimensional ODE which is determined by matrices that are intimately connected to those appearing in the associated Fokker-Planck equations.
\end{abstract}
	


\section{Introduction}\label{sec:intro}
\subsection{Background and the setting of the problem}\label{subsec:background}
Recent years have seen renewed interest in the long time behaviour of Fokker-Planck equations - particularly in various degenerate setting. For the sake of brevity in this short note, we shall avoid recalling the history and importance of such equations and refer the interested reader to works such as \cite{AEW18,AE,AMTU,ASS} for more details. 

Our work will focus on the following type of Fokker-Planck equation:
\begin{equation}\label{eq:fokkerplanck}
	\partial_t f(x,t)=-Lf(x,t):=\text{div}\pa{\D\nabla f(x,t)+\C xf(x,t)}, \quad\quad t>0, x\in\R^d,
\end{equation}
with appropriate initial conditions $f_0(x)$, where $\D$ and $\C$ are real valued constant matrices that satisfy the following conditions:
\begin{enumerate}[(A)]
	\item\label{item:cond_semipositive} $\D$ is positive semi-definite with
	$$1\leq r:=\operatorname{rank}\pa{\D} \leq d.$$
	\item\label{item:cond_positive_stability} All eigenvalues of $\C$ have positive real part (i.e. $\C$ is \textit{positive stable}).
	\item\label{item:cond_no_invariant_subspace_to_kernel} There exists no $\C^T$-invariant subspace of $\operatorname{ker}\pa{\D}$.
\end{enumerate}

As was shown in \cite{AE}, conditions \eqref{item:cond_semipositive}-\eqref{item:cond_no_invariant_subspace_to_kernel} guarantee the existence of a unique unit mass steady state to \eqref{eq:fokkerplanck}, $f_\infty(x)$, which is of Gaussian type. Furthermore, in their study \cite{AAS}, the authors have shown that there exists a \textit{linear reversible} transformation of the spatial variables under which the Fokker-Planck equation \eqref{eq:fokkerplanck} transforms into the same type of equation with drift and diffusion matrices $\wt{\C}$ and $\wt{\D}$ such that $\wt{\D}$ is diagonal and equals the symmetric part of $\wt{\C}$. As the connection between $\C$, $\D$ and $\wt{\C}$, $\wt{\D}$ is well established (see \cite{AAS,ASS} for instance) we will, for simplicity, assume from this point onwards that 
$$\D= \diag(d_1,\ldots,d_r, 0 \ldots 0) = \C_{s} = \frac{\C+\C^T}{2}.$$
We sometimes refer to this equation as the \textit{normalised Fokker-Planck equation}.
In this setting the unique steady state to our equation, which is, in fact, the equilibrium of the system, is given by the standard Gaussian
$$f_{\infty}(x) = \frac{1}{\pa{2\pi}^{\frac{d}{2}}}e^{-\frac{\abs{x}^2}{2}}.$$
 
Our goal in this work is to explore the convergence to equilibrium of (unit mass) solutions of \eqref{eq:fokkerplanck} under the $H^1$-based Fisher information. \\
The study of the convergence to equilibrium under the framework of \textit{relative entropies}, which are (usually non-linear) Lyaponuv functionals for the flow of the equation that ``measure'' how close the solution is to its equilibrium, is well established. Here we will assume familiarity with it and with the standard way to investigate such entropies - the so-called \textit{entropy method}. We refer the interested reader to \cite{AEW18,AEW22,AE,AMTU} amongst many other excellent introductions to this topic.

In many of the aforementioned studies, the investigation of the connection between the entropy and its production\footnote{The entropy production is defined as minus the functional that appears when we differentiate the entropy under the flow of the evolution.} relies on the so-called Bakry-Emery method (see \cite{BaEmH84, BaEmD85, BGL}), which involves differentiation in time of the entropy production. In the setting of  Fokker-Planck equations, issues with this method start to appear when the diffusion matrix $\D$ is degenerate and/or the drift matrix $\C$ has defects. In this case, the purely geometric functional inequality one looks for when using the Bakry-Emery method is not readily available\footnote{In fact, when $\D$ is degenerate one wouldn't be able to find the desired inequality that connects the entropy and its production.} - though one can circumvent this issue by allowing for \textit{time dependency} in the functional inequality, as was done by Monmarch\'e in \cite{MoH19}. 

It is still possible to obtain sharp long time behaviour of entropies of the form 
\begin{equation}\label{eq:def_Of_e_p}
	e_p\pa{f(t)|f_\infty} := \int_{\R^d}\psi_p\pa{\frac{f(x,t)}{f_\infty(x)}}f_\infty(x)dx
\end{equation}
with 
\begin{equation}\label{eq:def_of_psi_p}
	\psi_{p}(y):=\frac{y^p-p(y-1)-1}{p(p-1)},\qquad 1<p\leq 2,
\end{equation}
by using tools from Spectral Theory and hypercontractivity-like properties of the equation (see, for instance, \cite{AEW18}).

While giving an explicit convergence rate to the above essential entropies, the methodology used in \cite{AEW18}, circumventing the Bakry-Emery approach, seems less natural in the setting of the problem. In recent work, \cite{AEW22}, the authors of this note have defined a new notion of \textit{generalised Fisher information} that allowed them to bring back ideas that govern the Bakry-Emery method into their study of the equation. In particular, an essential ingredient of said study is the convergence to zero of the relative $2-$Fisher information, defined by 
\begin{equation}\nonumber
	I_2^{\PP}\pa{f|f_\infty} = \int_{\R^d}\nabla\pa{\frac{f(x)}{f_\infty(x)}}^T \PP \nabla\pa{\frac{f(x)}{f_\infty(x)}}f_\infty(x)dx,
\end{equation}
where $\PP$ is a given positive definite matrix, on various flow-invariant spaces that are connected to the spectral study of the Fokker-Planck operator.

In \cite{AEW22}, the authors focused on investigating $I_2^{\PP}$ with as little spectral information on the Fokker-Planck operator as possible. This resulted in the assumption that the diffusion matrix $\D$ is \textit{non-degenerate}. The goal of this short note is to remove this restriction by allowing for more spectral information to be considered.


\subsection{Main result}\label{subsec:main_result}

We remind the reader that throughout this work we assume that $\D$ is diagonal and equals the symmetric part of $\C$. In this setting the Fokker-Planck equation \eqref{eq:fokkerplanck} can be written as
\begin{equation}\label{eq:fpenormalized}
	\partial_t f(x,t)= -Lf(x,t)= \dive\left(f_\infty(x) \C \nabla \left(\frac{f\pa{x,t}}{f_\infty(x) }\right)\right).
\end{equation}

The study of the time evolution of the $2-$Fisher information, which is the entropy production of the $2-$entropy
$$e_2\pa{f|f_\infty} = \frac{1}{2}\int_{\R^d}\pa{f(x)-f_\infty(x)}^2f_\infty^{-1}(x)dx=\frac{1}{2}\norm{f-f_\infty}^2_{L^2\pa{\R^d,f_\infty^{-1}}},$$
is intimately related to the flow-invariant decomposition of the underlying Hilbert space, $L^2\pa{\R^d, f_\infty^{-1}}$. We thus begin with a few known facts about its structure in relation to the Fokker-Planck operator, $L$.
%
%

\begin{definition}
	Let $\alpha = (\alpha_i)\in \N_0^d$, with $\N_0=\N\cup\br{0}$, be an arbitrary multi-index, whose \textit{order} is defined to be $|\alpha| := \sum_{i=1}^d \alpha_i$. We define the $\alpha-$th \textit{Hermite functions} to be
	\begin{equation}\label{eq:galpha}
		h_\alpha(x):= (-1)^{|\alpha|} \partial^\alpha f_\infty (x),
	\end{equation}
	where $\partial^\alpha := \partial_{x_1}^{\alpha_1}\cdots\partial^{\alpha_d}_{x_d}$. For each $m\in\N_0$ we define
	\begin{equation*}
		V_m:= \linspan\{h_\alpha : \alpha \in \N^d_0, |\alpha| = m\}\subseteq L^2(\R^d, f_\infty^{-1}).
	\end{equation*}
\end{definition}

\begin{remark}\label{rem:hermite_poly}
	
	It is straightforward to see that one can write $h_\alpha(x)$ as $H_{\alpha}(x)f_\infty(x)$, where $H_{\alpha}(x)$ is a polynomial of degree $\abs{\alpha}$. These polynomial are known as the \it{Hermite polynomials}. 
\end{remark}

The following properties of $\br{V_m}_{m\in\N_0}$ and connections between these spaces and the spectrum of $L$ were shown in \cite{AE}:

\begin{theorem}\label{thm:spectraldecomposition}
	Assume that the drift and diffusion matrices $\C$ and $\D$ satisfy \conA. Then
	\begin{enumerate}[(i)]
		\item $\br{V_m}_{m\in \N_0}$ are mutually orthogonal finite dimensional spaces in $L^2\pa{\R^d, f_\infty ^{-1}}$.
		\item \begin{equation}\nonumber
			L^2\pa{\R^d,f_\infty^{-1}}= \bigoplus_{m\in\N_0}V_m.
		\end{equation}
		\item $V_m$ are invariant under $L$ and its adjoint.
		\item The spectrum of $L$ satisfies
		\begin{equation}\nonumber
			\sigma\pa{L}=\bigcup_{m\in\N_0}\sigma\pa{L\vert_{V_m}},
		\end{equation}
		and
		\begin{equation}\nonumber
			\sigma\pa{L\vert_{V_m}}=\br{-\sum_{i=1}^d \alpha_i \lambda_i \ \Big|\ \alpha_1 ,\dots,\alpha_d\in\N_0, \abs{\alpha}=m},
		\end{equation}
		where $\br{\lambda_j}_{j=1,\dots,d}$ are the eigenvalues (with possible multiplicity) of the matrix $\C$. The eigenfunctions (or eigenfunctions and generalised eigenfunctions in the case $\C$ is defective) of $\br{L\vert_{V_m}}_{m\in \N_0}$ form a basis for $L^2\pa{\R^d,f_\infty^{-1}}$.
	\end{enumerate}
\end{theorem}

The orthogonal decomposition of $L^2\pa{\R^d,f_\infty^{-1}}$ insinuates that the convergence to equilibrium of a unit mass solution, i.e. the convergence to $f_\infty$ (which spans $V_0=\ker\pa{L}$), is governed by the eigenvalues of $\br{L\vert_{V_m}}_{m\in\N}$ whose real part is closest to zero, which is clearly attained in $V_1$. This idea is the basis of the study performed in \cite{AEW18} which was then extended to the following elegant result in \cite{ASS}:

\begin{theorem}\label{thm:propagator}
	Consider the normalised Fokker-Planck equation \eqref{eq:fokkerplanck} and assume that \conA\ hold. Then 
	\begin{equation*}
		\norm{e^{-{L}t}}_{\mathscr{B}\pa{V_0^{\perp}}} = \|e^{-\C t}\|_2, \quad \forall t\geq 0,
	\end{equation*}
	where $\mathscr{B}\pa{\mathcal{X}}$ is the space of all bounded linear operators from $\mathcal{X}$ to itself with the usual operator norm, and $\norm{\cdot}_2$ is the matrix norm with respect to the Euclidean norm on $\R^d$.
\end{theorem}

Theorem \ref{thm:propagator} automatically implies that 
$$e_2\pa{f(t)|f_\infty} \leq C\pa{1+t}^{2n}e^{-\mu t}e_2\pa{f_0|f_\infty},$$
where $C>0$ is a fixed constant that depends only on $\C$ and the dimension, 
\begin{equation}\label{eq:def_of_mu}
	\mu := \min \{ \re \pa{\lambda} : \lambda \text{ is an eigenvalue of } \C\}>0,
\end{equation}
and $n$ is the largest defect amongst the eigenvalues whose real part is $\mu$.

Theorem \ref{thm:propagator}, however, does not immediately give the same result for the associated $2-$Fisher information. The reason behind this is that the evolution of the integrand appearing in the Fisher information
 is \textit{not} governed by $L$. 
 
 Indeed, denoting by  
\begin{equation}\label{eq:Jt}
	\bm{J}(x,t):= \pa{J_1(x,t),\dots,J_d(x,t)}^{T}:=	f_\infty(x)\nabla \pa{\frac{f(x,t)}{f_\infty(x)}}  
\end{equation}
we find that
\begin{equation}
	\begin{aligned}\label{eq:fitonorm}
		I_2^{\II}(f(t)|f_\infty) =& \int_{\R^d} \abs{ \nabla \pa{\frac{f(x,t)}{f_\infty\pa{x}}}}^2 f_\infty(x) dx \\
		&= \sum_{j=1}^d \norm{J_j(t)}^2_{\h}
		=\norm{\bm{J}(t)}^2_{\pa{\h}^d}.
	\end{aligned}
\end{equation}
Under the assumption that $f$ is a solution to \eqref{eq:fpenormalized} we have that
	\begin{equation}
	\begin{aligned}
		\partial_t \bm{J}\pa{x,t}&= f_\infty(x) \nabla \pa{\frac{\partial_t f(x,t)}{f_\infty(x)}}= f_\infty(x) \nabla \pa{\frac{\dive\left(f_\infty(x) \C \nabla \left(\frac{f(x,t)}{f_\infty(x) }\right)\right)}{f_\infty(x)}} \\
		&= f_\infty(x) \nabla \pa{\frac{\dive\left( \C \bm{J}(t)\right)}{f_\infty(x)}} = (\nabla+ x)\pa{\dive(\C \bm{J}(t))}. \label{eq:Jprop}
	\end{aligned}
\end{equation}

Nevertheless we will be able to obtain the following, which is the main result of this short note:

\begin{theorem}
	\label{thm:fulldecay}
	Consider the normalised Fokker-Planck equation \eqref{eq:fokkerplanck} where \conA are satisfied. Let $f_0\in  \h \cap \pa{ \bigoplus_{k=0}^{m-1} V_k}^\perp$ for some $m\in\N$ be such that $I^{\II}_2\pa{f_0|f_\infty}<\infty$. Then, the solution $f(t)$ to the Fokker-Planck equation satisfies
	\begin{equation}\label{eq:fidecay}
		I^{\II}_2\pa{f(t)|f_\infty} \leq \norm{e^{-\C t}}^{2m}_{2} I^{\II}_2\pa{f_0|f_\infty}, \quad \forall t\geq 0.
	\end{equation}
	In particular, there exists a constant $C_{m}>0$, depending only on $\C$, $m$, and the dimension, such that 
		\begin{equation}\label{eq:fiexactdecay}
		I^{\II}_2\pa{f(t)|f_\infty} \leq C_{m}\pa{1+t}^{2nm}e^{-2m\mu t} I^{\II}_2\pa{f_0|f_\infty}, \quad \forall t\geq 0,
	\end{equation}
	where $\mu$ is defined in \eqref{eq:def_of_mu} and $n$ is the largest defect associated to the eigenvalues of $\C$ whose real part is $\mu$.
\end{theorem}
\begin{remark}\label{rem:extension_to_PP}
	It is worth to note that, since for any positive definite matrix $\PP$ we have that 
	$$\mathrm{p}_{\mathrm{min}} I_2^{\II}\pa{f|f_\infty} \leq I_2^{\PP}\pa{f|f_\infty} \leq \mathrm{p}_{\mathrm{max}} I_2^{\II}\pa{f|f_\infty}$$
	where $\mathrm{p}_{\mathrm{min}}$ and $\mathrm{p}_{\mathrm{max}}$ are the minimal and maximal eigenvalues of $\PP$ respectively, we can easily extend Theorem \ref{thm:fulldecay} to the general $I_2^{\PP}$. 
\end{remark}

\begin{remark}
	The decay estimate in \eqref{eq:fidecay} is sharp. Equality can be reached by an explicit choice of the initial condition $f_0$ in the finite dimensional subspace $V_m$. 
\end{remark}

The idea of the proof of Theorem \ref{thm:fulldecay} is to use the orthogonality of the $V_m-$s together with the fact, shown in \cite{ASS}, that the Fokker-Planck operator on each finite dimensional subspace $V_m$ behaves like an $m-$times tensorisation of the drift matrix $\C$. 

%

\section{Sharp decay of the $2-$Fisher information}\label{sec:decay}

In order to prove our main result, we note without proof a few simple well-known properties of the Hermite functions see, for instance, \cite{AS1983}.


\begin{lemma}\label{lem:Hermite_properties}
	Consider the Hermite functions $h_\alpha$ defined in \eqref{eq:galpha}. Then:
	\begin{enumerate}[(i)]
		\item\label{item:Hermite_poly_grad} For any multi-index $\alpha$ and any $j=1,\dots,d$ we have that
		\begin{equation*}
			\partial_{x_j}h_{\alpha}(x) = - h_{\alpha+\bm{e}_j}(x),
		\end{equation*}
		where $\br{\bm{e}_j}_{j=1,\dots,d}$ is the standard basis of $\R^d$.
		\item\label{item:Hermite_function_norm} For any multi-indexes $\alpha$ and $\beta$ we have that
		$$\inner{h_\alpha,h_\beta}_{L^2\pa{\R^d,f_\infty^{-1}}}=\alpha! \;\delta_{\alpha,\beta}\,,$$
	\end{enumerate}
	where $\alpha! := \prod_{i=1}^d \alpha_i!\,$.
\end{lemma}

Combining Theorem \ref{thm:spectraldecomposition} with Lemma \ref{lem:Hermite_properties} we see that any $f\in L^2\pa{\R^d,f_\infty^{-1}}$ can be written with respect to the above basis as
\begin{equation}\label{eq:hermite_decomp}
	f=\sum_{k=0}^\infty \sum_{\abs{\alpha}=k} d_{\alpha}h_{\alpha}\,,
\end{equation}
where
\begin{equation}\label{eq:hermite_coeff}
	d_\alpha = \frac1{\alpha!} \inner{f,h_{\alpha}}_{L^2\pa{\R^d,f_\infty^{-1}}}.
\end{equation}
This observation, as well as the regularisation properties of the Fokker-Planck equation, give us the following:

\begin{lemma}\label{lem:derivative_coefficients}
	Consider the normalised Fokker-Planck equation \eqref{eq:fpenormalized} and let $f(t)$ be the solution to it with initial condition $f_0\in L^2\pa{\R^d,f_{\infty}^{-1}}$. Then for any $t>0$ we have that $f(t)\in L^2\pa{\R^d,f_{\infty}^{-1}}\cap C^{\infty}\pa{\R^d}$. Furthermore, $\partial_{x_j} f(x,t)+x_j f\pa{x,t}\in L^2\pa{\R^d,f_{\infty}^{-1}}$ for any $j=1,\dots,d$ and $t>0$, and
	\begin{equation}\label{eq:derivative_coeff}
		\inner{\partial_{x_j}f(x,t)+x_{j}f(x,t),h_{\alpha}(x)}_{L^2\pa{\R^d,f_{\infty}^{-1}}}=\inner{f(t),h_{\alpha+\bm{e}_j}}_{L^2\pa{\R^d,f_{\infty}^{-1}}}.
	\end{equation}
	Consequently, writing
	\begin{equation}\label{eq:L2_expansion}
		f(t)=\sum_{k=0}^\infty \sum_{\abs{\alpha}=k} d_{\alpha}(t)h_{\alpha},
	\end{equation}
	we have that for any $t>0$
	\begin{equation}\label{eq:hermite_derivative_decomp}
		\partial_{x_j} f(x,t)+x_j f\pa{x,t} = \sum_{k=0}^\infty \sum_{\abs{\alpha}=k} \pa{\alpha_j+1}d_{\alpha+\bm{e}_j}(t)h_{\alpha}(x),
	\end{equation}
	where $j=1,\dots, d$.
\end{lemma}

\begin{remark}\label{rem:I_exists_at_0}
	We would like to point out that Lemma \ref{lem:derivative_coefficients} only requires that the initial datum, $f_0$, will be in $L^2\pa{\R^d,f_{\infty}^{-1}}$. There is no requirement that the $2-$Fisher information would be finite at time $t=0$. This is exactly why we can infer that $\partial_{x_j} f(x,t)+x_j f\pa{x,t}\in L^2\pa{\R^d,f_{\infty}^{-1}}$ for any $j=1,\dots,d$ \textit{only} for $t>0$. 
\end{remark}

\begin{proof}[Proof of Lemma \ref{lem:derivative_coefficients}]
	The first part of the lemma is a standard result in the study of Fokker-Planck equations of the form \eqref{eq:fokkerplanck} where \conA are satisfied (see, for instance, Corollary 2.6 in \cite{AE}). Theorem 4.8 of the same study, or Theorem A.12 in \cite{Villani2009}, show the instantaneous generation of a finite $2-$Fisher information, i.e. that for any $t>0$ we have that
	$$\nabla\pa{\frac{f(x,t)}{f_\infty(x)}} =\frac{\nabla f\pa{x,t}+xf(x,t)}{f_\infty(x)} \in L^2\pa{\R^d,f_\infty}^d,$$
	or equivalently that $\nabla f\pa{x,t}+xf(x,t)\in L^2\pa{\R^d,f_\infty^{-1}}^d$. 
	
	Identity \eqref{eq:derivative_coeff} follows from a simple integration by parts 
	and part \eqref{item:Hermite_poly_grad} of Lemma \ref{lem:Hermite_properties}:
	$$	\inner{\partial_{x_j}f(x,t)+x_{j}f(x,t),h_{\alpha}(x)}_{L^2\pa{\R^d,f_{\infty}^{-1}}} = \int_{\R^d} \partial_{x_j}\pa{\frac{f(x,t)}{f_\infty(x)}}h_\alpha(x)dx$$
	$$=-\int_{\R^d}\frac{f\pa{x,t}}{f_\infty(x)}\partial_{x_j}h_{\alpha}(x) dx=\inner{f(t),h_{\alpha+\bm{e}_j}}_{L^2\pa{\R^d,f_{\infty}^{-1}}} .$$
	The integration by parts is justified\footnote{For $d=1$ this justification follows from the density of $C_c^\infty(\R)$ in the weighted Sobolev space $H^1(\R,f_\infty)$ \cite[\S1]{Zh98}. But the authors are not aware of an analogous result for $\R^d$, $d>1$.} since for any $\alpha\in \N_0^d$ and any $t>0$ we have that 
	$$\abs{x}^{\abs{\alpha}}\abs{f\pa{x,t}} \underset{\abs{x}\to\infty}{\longrightarrow}0.$$
	This follows from the explicit formula for the solution of the Fokker-Planck equation (see \cite{AE})	
	\begin{equation}\nonumber
		f(x,t)=\int_{\R^d}G(x-e^{-\C t}y,t)f_0(y)dy,
	\end{equation}
	where 
	\begin{equation*}
		G(z,t)=\frac{1}{\pa{2\pi}^{\frac{d}{2}}\sqrt{\det \W(t)}}e^{-\frac{1}{2}z^T \W(t)^{-1}z},\quad \W(t):=2\int_{0}^t e^{-\C s}\D e^{-\C^T s}ds.
	\end{equation*} 
	Indeed, since $f_0\in L^2\pa{\R^d,f_{\infty}^{-1}}$ we see that if we define $f_{0,k}:=\abs{x}^{k}f_0(x)$ for a given $k\in\N_0$, then $f_{0,k}\in L^1\pa{\R^d,dx}$ and 
	$$\abs{x}^{k}\abs{f\pa{x,t}} \leq 2^{k}\Big(\int_{\R^d} \abs{x-e^{-\C t}y}^k G(x-e^{-\C t}y,t)\abs{f_0(y)}dy$$
	$$+ C\int_{\R^d} G(x-e^{-\C t}y,t)\abs{f_{0,k}(y)}dy\Big)<\infty,$$
	where the second line above follows from the fact that $\C$ is positively stable.\\
	This implies that for any $\alpha\in \N_0^d$ and any $t>0$ we have that 
	$$0 \leq \abs{x}^{\abs{\alpha}}\abs{f\pa{x,t}}  \leq \frac{\sup_{z\in\R^d}\abs{z}^{\abs{\alpha}+1}\abs{f\pa{z,t}}}{\abs{x}}\underset{\abs{x}\to\infty}{\longrightarrow}0.$$
	We conclude \eqref{eq:hermite_derivative_decomp} from the fact that 
	$$\frac1{\alpha !}\inner{	\partial_{x_j} f(x,t)+x_j f\pa{x,t} , h_\alpha(x)}_{\h}=\pa{\alpha_j+1}\frac{\inner{f(t),h_{\alpha+\bm{e}_j}}_{L^2\pa{\R^d,f_{\infty}^{-1}}}}{\pa{\alpha+\bm{e}_j}!}.$$
\end{proof}

Lemma \ref{lem:derivative_coefficients} gives us the ability to connect between the $L^2$ expansion of a function and its $2-$Fisher information in the following way:

\begin{theorem}\label{thm:norm_and_fisher}
	Consider the normalised Fokker-Planck equation \eqref{eq:fpenormalized} and let $f(t)$ be the solution to it with initial condition $f_0\in L^2\pa{\R^d,f_{\infty}^{-1}}\cap \pa{ \bigoplus_{k=0}^{m-1} V_k}^\perp$ with $m\in \N$. Writing
	$$f(t)=\sum_{k=0}^\infty \sum_{\abs{\alpha}=k} d_{\alpha}(t)h_{\alpha}$$
		we have that for any $t>0$
	\begin{equation}\label{eq:norm_and_fisher}
		I_2^{\II}\pa{f(t)|f_\infty} = \sum_{k=m}^\infty k\sum_{\abs{\alpha}=k} \alpha!\;d^2_{\alpha}(t).
	\end{equation}
\end{theorem}
\begin{proof} 
	Since $f_0\in \pa{ \bigoplus_{k=0}^{m-1} V_k}^\perp$, which is invariant under the flow of the equation, we find that $f(t)\in \pa{ \bigoplus_{k=0}^{m-1} V_k}^\perp$ for all $t\geq 0$. Consequently if we write 
	$$f(t)=\sum_{k=0}^\infty \sum_{\abs{\alpha}=k} d_{\alpha}(t)h_{\alpha}$$
	we find that $d_{\alpha}(t)=0$ for any multi-index $\alpha$ such that $\abs{\alpha}<m$. As we saw in Lemma \ref{lem:derivative_coefficients}, for any $j=1,\dots,d$ we have that for any $t>0$
	$$J_j(x,t)=f_\infty(x)\partial_{x_j} \pa{\frac{f(x,t)}{f_\infty(x)}}=\partial_{x_j}f(x,t)+x_jf\pa{x,t} $$
	$$= \sum_{k=0}^\infty \sum_{\abs{\alpha}=k}\pa{\alpha_j+1} d_{\alpha+\bm{e}_j}(t)h_{\alpha}(x)=\sum_{k=m}^\infty \sum_{\abs{\alpha}=k} \alpha_jd_{\alpha}(t)h_{\alpha-\bm{e}_j}(x).$$
	Using \eqref{eq:fitonorm}
	and the orthogonality of $\br{h_\alpha}$, we find that
	$$I_2^{\II}(f(t)|f_\infty) =\sum_{j=1}^d\sum_{k=m}^\infty \sum_{\abs{\alpha}=k} \alpha_j^2\pa{\alpha - \bm{e}_j}!\;d^2_{\alpha}(t) =\sum_{j=1}^d\sum_{k=m}^\infty \sum_{\abs{\alpha}=k} \alpha_j \alpha!\;d^2_{\alpha}(t)$$
	$$=\sum_{k=m}^\infty \sum_{\abs{\alpha}=k} \abs{\alpha} \alpha!\;d^2_{\alpha}(t)=\sum_{k=m}^\infty k\sum_{\abs{\alpha}=k} \alpha!\;d^2_{\alpha}(t),$$
	which is the desired result. 
\end{proof}

We are missing only one additional ingredient to be able to prove our main theorem --- the evolution of the expansion coefficients $\br{d_{\alpha}}$. This can be found by combining Proposition 5.3 and Theorem 6.1 in \cite{ASS}:
\begin{theorem}\label{thm:vmdecay}
	Consider the normalised Fokker-Planck equation \eqref{eq:fpenormalized} and let $f(t)$ be the solution to it with initial condition $f_0\in L^2\pa{\R^d,f_{\infty}^{-1}}$. Then for any $t\geq0$ and any $m\in\N$ we have that
	\begin{equation}\label{eq:l2decay}
		\sum_{\abs{\alpha}=m}\alpha!\;d^2_{\alpha}(t) \leq \norm{e^{-\C t}}_2^{2m} \sum_{\abs{\alpha}=m}\alpha!\;d^2_{\alpha}(0) .
	\end{equation}
\end{theorem}

We conclude our short note with the proof of  our main theorem.
\begin{proof}[Proof of Theorem \ref{thm:fulldecay}]
	Fix some $t_0>0$. Using Theorem \ref{thm:norm_and_fisher} together with a time shifted version of Theorem \ref{thm:vmdecay} we find that for $t\geq t_0$
	$$I_2^{\II}\pa{f(t)|f_\infty} = \sum_{k=m}^\infty k\sum_{\abs{\alpha}=k} \alpha!\;d^2_{\alpha}(t) \leq \sum_{k=m}^\infty k\norm{e^{-\C \pa{t-t_0}}}_2^{2k} \sum_{\abs{\alpha}=k} \alpha!\;d^2_{\alpha}(t_0).$$
	As was discussed in \cite{ASS}, $\norm{e^{-\C  \pa{t-t_0}}}_2 \leq 1$ and as such, using \eqref{eq:norm_and_fisher} again, we find that for $t\geq t_0$
		$$I_2^{\II}\pa{f(t)|f_\infty} \leq \norm{e^{-\C  \pa{t-t_0}}}_2^{2m}  \sum_{k=m}^\infty k\sum_{\abs{\alpha}=k} \alpha!\;d^2_{\alpha}(t_0)=\norm{e^{-\C  \pa{t-t_0}}}_2^{2m}I_2^{\II}\pa{f(t_0)|f_\infty}.$$
	Since $I_2^{\II}\pa{f_0|f_\infty}< \infty$, taking $t_0$ to zero in the above increases the right hand side as both factors $\norm{e^{-\C\pa{t-t_0}}}_2$ and $I_2^\II\pa{f\pa{t_0}|f_\infty}$ increase when $t_0$ goes to zero. The latter follows from Proposition 4.5 in \cite{AE} (using $\PP:=\II$). This concludes the proof.
\end{proof}

\bibliographystyle{plain}

\begin{thebibliography}{99}
	
	
	\bibitem{AS1983}
	M. Abramowitz, I. A. Stegun, \textit{Handbook of mathematical functions with formulas, graphs, and mathematical tables} (1992).  
	
	\bibitem{AAS} 
	F. Achleitner, A. Arnold, D. St{\"{u}}rzer, \textit{Large-Time Behavior in Non-Symmetric Fokker-Planck Equations}.
	Rivista di Matematica della Universit\`a di Parma \textbf{6} (2015) 1--68. 
	
	
	\bibitem{AEW18} A. Arnold, A. Einav and T. W\"ohrer. \textit{On the Rate of Decay to Equilibrium in Degenerate and Defective Fokker-Planck Equations}. J. Differential Equations, \textbf{264} (2018), No. 11, 6843--6872. 
	
	\bibitem{AEW22}
	A. Arnold, A. Einav, T. W\"ohrer, \textit{Generalised Fisher Information in Defective Fokker-Planck Equations}. Preprint (2022). 
	
	\bibitem{AE}
	A. Arnold, J. Erb, \textit{Sharp Entropy Decay for Hypocoercive and Non-Symmetric Fokker-Planck Equations with Linear Drift}. Preprint. https://arxiv.org/abs/1409.5425
	
	
	
	
	
	
	
	
	
	\bibitem{AMTU}
	A. Arnold, P. Markowich, G. Toscani, A. Unterreiter, \textit{On convex Sobolev inequalities and the rate of convergence to equilibrium for Fokker--Planck type equations}, Communications in Partial Differential Equations \textbf{26} (2001), 43--100.
	
	\bibitem{ASS}
	A. Arnold, B. Signorello, C. Schmeiser, \textit{Propagator norm and sharp decay estimates for Fokker-Planck equations with linear drift}, Comm. Math. Sc. \textbf{20}, No. 4 (2022) 1047-1080
	
	
	\bibitem{BaEmH84}
	D. Bakry, M. Emery, \textit{Hypercontractivit\'e de semi-groupes de diffusion}, C. R. Acad. Sci. Paris S\'er. I Math. \textbf{299} (1984), 775--778.
	
	\bibitem{BaEmD85}
	D. Bakry, M. Emery, \textit{Diffusions hypercontractives}, S\'eminaire de probabilt\'es de Strasbourg \textbf{19} (1985), 177--206.
	
	
	\bibitem{BGL}
	D. Bakry, I. Gentil, M. Ledoux, \textit{Analysis and geometry of Markov diffusion operators}, Springer Science \& Business Media, (2013).
	
	
	\bibitem{MoH19}
	P. Monmarch\'e, \textit{Generalized $\Gamma$ Calculus and Application to Interacting Particles on a Graph}, Potential Anal. \textbf{3} (2019), 439--466.
	
	\bibitem{Villani2009}
	C. Villani, \textit{Hypocoercivity}, Mem. Amer. Math. Soc. \textbf{202} (2009), no. 950. 
	%
	%
		\bibitem{Zh98}
V. V. Zhikov, \textit{Weighted Sobolev spaces}, 
Sb. Math. \textbf{189} (1998) 1139-1170.

	
	
\end{thebibliography}

\end{document}